\documentclass[11pt]{amsart}

\usepackage{amsmath,amsfonts,amssymb,amsthm}

\title[A plane model with two Galois points]{A criterion for the existence of a plane model with two inner Galois points for algebraic curves}
\author{Kazuki Higashine}

\thanks{The author was partially supported by Grant-in-Aid for JSPS Research Fellows (No. 20J12384).}
\address[Kazuki Higashine]{Graduate School of Science and Engineering, Yamagata University, Kojirakawa-machi 1-4-12, Yamagata 990-8560, Japan}
\email{s182102d@st.yamagata-u.ac.jp}

\newtheorem{theorem}{Theorem}
\newtheorem{proposition}[theorem]{Proposition}
\newtheorem{lemma}[theorem]{Lemma}
\newtheorem{fact}[theorem]{Fact}

\theoremstyle{definition}
\newtheorem{remark}[theorem]{Remark}

\begin{document}
\begin{abstract}
A criterion for the existence of a plane model with two non-smooth Galois points for algebraic curves is presented, which is a generalization of Fukasawa's criterion for two smooth Galois points. Owing to this generalized criterion, multiplicities and order sequences at Galois points can be described in detail.
\end{abstract}
\maketitle

\section{Introduction}
Let $k$ be an algebraically closed field of characteristic $p\geq0$, and let $C\subset \mathbb{P}^{2}$ be an irreducible (possibly singular) plane curve of degree $d\geq 4$ over $k$. For points $P$, $Q\in \mathbb{P}^{2}$ with $P\not= Q$, the line passing through $P$, $Q$ is denoted by $\overline{PQ}$. We consider the projection $\pi_{P}:C\dashrightarrow \mathbb{P}^{1}; Q\mapsto \overline{PQ}$ with the center $P\in \mathbb{P}^2$. If the field extension $k(C)/\pi_{P}^{\ast}k(\mathbb{P}^{1})$ of function fields induced by $\pi_{P}$ is Galois, then $P$ is called a Galois point for $C$. This notion was introduced in 1996 by Yoshihara (see \cite{FukasawaDedicata2009, MiuraYoshiharaAlgebra2000, YoshiharaAlgebra2001}). For a Galois point $P$, the associated Galois group $G_{P}={\rm Gal}(k(C)/\pi_{P}^{\ast}k(\mathbb{P}^{1}))$ is called a Galois group at $P$. Furthermore, a Galois point $P$ is called a smooth Galois point (resp. a non-smooth Galois point, an inner Galois point, an outer Galois point) if $P$ is a smooth point of $C$ (resp. a singular point of $C$, a point contained in $C$, a point not contained in $C$), after \cite{MiuraAustral.2000, MiuraAlgebra2005, TakahashiNihonkai2005}. 

In 2016, Fukasawa \cite{FukasawaAlgebra2018} presented a criterion for the existence of a birational embedding of a smooth projective curve into a projective plane with two smooth Galois points and obtained new examples of plane curves with two smooth Galois points by using this criterion. On the other hand, there have been some known examples of plane curves with two or more non-smooth Galois points. For example, the Ballico-Hefez curve (\cite[Theorem 1]{FukasawaFinite2013}), some self-dual curves (\cite[Theorem 17]{HayashiYoshiharaGeometry2013}), the (plane model of) Giulietti-Korchm\'{a}ros curve (\cite[Theorem 2]{FukasawaHigashineFinite2019}), the ($q^{3}, q^{2}$)-Frobenius nonclassical curve (\cite[Theorem 1]{BorgesFukasawaFinite2020}), and the Artin-Schreier-Mumford curve (proof of \cite[Theorem 1]{FukasawaarXiv2005.10073}). However, only few research studies have focused on non-smooth Galois points. Takahashi \cite{TakahashiNihonkai2005} studied a plane quintic curve with a double-point $P$ and determined defining equations when $P$ is a Galois point. This is the only study that focused on a non-smooth Galois point known so far. In order to study non-smooth Galois points systematically, it is good to have a criterion for non-smooth Galois points. 

In this article, we extend Fukasawa's criterion \cite[Theorem 1]{FukasawaAlgebra2018} to all cases with two (possibly non-smooth) Galois points. Let $X$ be a (reduced irreducible) smooth projective curve over $k$, and let $k(X)$ be its function field. The full automorphism group of $X$ is denoted by ${\rm Aut}(X)$. For a finite subgroup $G\subset {\rm Aut}(X)$ and a point $P\in X$, the stabilizer of $P$ in $G$ (resp. the orbit of $P$ under $G$) is denoted by $G(P)$ (resp. $G\cdot P$). Furthermore, the quotient curve of $X$ by $G$, that is, the smooth projective curve corresponding to the fixed field of $k(X)$ by $G$, is denoted by $X/G$. The following theorem is the main theorem.

\begin{theorem}\label{theorem}
Let $G_{1}, G_{2}$ be finite subgroups of ${\rm Aut}(X)$ and let $P_{1}, P_{2}$ be different points of $X$. We put $\mathbb{O} =(G_{1}\cdot P_{2})\cup (G_{2}\cdot P_{1})$. We consider the two divisors on $X;$
\begin{itemize}
\item ${\rm Bs}_{P_{1}}=|G_{2}(P_{1})|\sum_{Q\in \mathbb{O}\setminus G_{1}\cdot P_{2}}Q+(|G_{2}(P_{1})|-|G_{1}(P_{2})|)\sum_{R\in G_{1}\cdot P_{2}\cap G_{2}\cdot P_{1}}R$,
\item ${\rm Bs}_{P_{2}}=|G_{1}(P_{2})|\sum_{S\in \mathbb{O}\setminus G_{2}\cdot P_{1}}S$.
\end{itemize}
Then the four conditions
\begin{itemize}
\item[(a)] $X/G_{1}\cong \mathbb{P}^{1}, X/G_{2}\cong \mathbb{P}^{1}$,
\item[(b)] $G_{1}\cap G_{2}=\{ 1\}$,
\item[(c)] ${\rm Bs}_{P_{1}}\geq P_{1}, {\rm Bs}_{P_{2}}\geq P_{2}$, and
\item[(d)] ${\rm Bs}_{P_{1}}+\sum_{\sigma \in G_{1}}\sigma (P_{2})={\rm Bs}_{P_{2}}+\sum_{\tau \in G_{2}}\tau (P_{1})$,
\end{itemize}
are satidfied, if and only if there exists a birational embedding $\varphi :X\to \mathbb{P}^{2}$ of degree $|G_{1}|+{\rm deg}({\rm Bs}_{P_{1}})=|G_{2}|+{\rm deg}({\rm Bs}_{P_{2}})$ such that $\varphi (P_{1})$, $\varphi (P_{2})$ are different inner Galois points for $\varphi (X)$, and $G_{\varphi (P_{1})}, G_{\varphi (P_{2})}$ coincide with $G_{1}, G_{2}$ respectively, and $\overline{\varphi(P_{1})\varphi(P_{2})}$ is not a tangent line at $\varphi (P_{2})$.
\end{theorem}

By Theorem \ref{theorem} and its proof, we have the following.

\begin{proposition}\label{caution2}
Let $\varphi$ be as in Theorem \ref{theorem}, and $\Lambda$ be the linear system on $X$ corresponding to the morphism $\varphi$. Then the following hold.
\begin{itemize}
\item[(1)] The multiplicity $m_{\varphi (P_{1})}$ of $\varphi (X)$ at $\varphi (P_{1})$ is equal to $|G_{2}(P_{1})|\cdot |\mathbb{O}\setminus G_{1}\cdot P_{2}|+(|G_{2}(P_{1})|-|G_{1}(P_{2}|)\cdot |G_{1}\cdot P_{2}\cap G_{2}\cdot P_{1}|$.
\item[(2)] The multiplicity $m_{\varphi (P_{2})}$ of $\varphi (X)$ at $\varphi (P_{2})$ is equal to $|G_{1}(P_{2})|\cdot |\mathbb{O}\setminus G_{2}\cdot P_{1}|$.
\item[(3)] If $P\in \varphi^{-1}(\varphi (P_{1}))$ and $\overline{\varphi(P_{1})\varphi(P_{2})}$ is not the osculating line at $P$, then the second $(\Lambda, P)$-order coincides with $|G_{2}(P_{1})|$.
\item[(4)] If $P\in \varphi^{-1}(\varphi (P_{1}))$ and $\overline{\varphi(P_{1})\varphi(P_{2})}$ is the osculating line at $P$, then the second $(\Lambda, P)$-order coincides with $|G_{2}(P_{1})|-|G_{1}(P_{2})|$, and the third $(\Lambda, P)$-order coincides with $|G_{2}(P_{1})|$.
\item[(5)] If $P\in \varphi^{-1}(\varphi (P_{2}))$, then the second $(\Lambda, P)$-order coincides with $|G_{1}(P_{2})|$.
\end{itemize}
\end{proposition}

To explain the usefulness of Theorem \ref{theorem}, we apply our criterion to rational curves.

\begin{theorem}\label{example}
There exist the following birational embeddings $\varphi :\mathbb{P}^{1}\to \mathbb{P}^{2}$.
\begin{itemize}
\item[(1)] $p=3$, ${\rm deg}(\varphi (\mathbb{P}^{1}))=14$ and there exist two non-smooth Galois points $\varphi (P_{1})$ and $\varphi (P_{2})\in \varphi (\mathbb{P}^{1})$ such that $m_{\varphi (P_{1})}=8, m_{\varphi (P_{2})}=4, G_{\varphi (P_{1})}\cong {\rm AGL}(1, \mathbb{F}_{3})$, $G_{\varphi (P_{2})}\cong {\bf D}_{5}$, and $\overline{\varphi (P_{1})\varphi (P_{2})}$ is not a tangent line at $\varphi (P_{1})$ and $\varphi (P_{2})$.
\item[(2)] $p\not= 2, 5$, ${\rm deg}(\varphi (\mathbb{P}^{1}))=16$ and there exist two non-smooth Galois points $\varphi (P_{1})$ and $\varphi (P_{2})\in \varphi (\mathbb{P}^{1})$ such that $m_{\varphi (P_{1})}=11, m_{\varphi (P_{2})}=4, G_{\varphi (P_{1})}\cong \mathbb{Z}/5\mathbb{Z}, G_{\varphi (P_{2})}\cong {\bf A}_{4},$ and $\overline{\varphi (P_{1})\varphi (P_{2})}$ is a tangent line at $\varphi (P_{1})$ but not at $\varphi (P_{2})$. 
\item[(3)] $p\not= 2, 5$, ${\rm deg}(\varphi (\mathbb{P}^{1}))=28$ and there exist two non-smooth Galois points $\varphi (P_{1})$ and $\varphi (P_{2})\in \varphi (\mathbb{P}^{1})$ such that $m_{\varphi (P_{1})}=23, m_{\varphi (P_{2})}=4, G_{\varphi (P_{1})}\cong \mathbb{Z}/5\mathbb{Z}, G_{\varphi (P_{2})}\cong {\bf S}_{4}$, and $\overline{\varphi (P_{1})\varphi (P_{2})}$ is a tangent line at $\varphi (P_{1})$ but not at $\varphi (P_{2})$.
\end{itemize}
\end{theorem}

In more detail, we can calculate order sequences at each point contained in the support of the divisor $\varphi^{\ast}\overline{\varphi (P_{1})\varphi (P_{2})}$ (see Remark \ref{caution3}).

\section{Preliminaries}
We recall some notations and facts. Let $\varphi : X\to \mathbb{P}^{2}$ be a morphism, which is birational onto its image. Such morphism $\varphi$ is called a birational embedding of $X$ to $\mathbb{P}^{2}$. First, we recall the notion of order sequences (see \cite[Capter 7]{hkt}). For a line $L\subset \mathbb{P}^{2}$, the intersection divisor of $\varphi (X)$ and $L$ on $X$ is denoted by $\varphi^{\ast}L$. Note that $\Lambda =\{ \varphi^{\ast}L\mid L{\rm\ is\ a\ line\ contained\ in\ }\mathbb{P}^{2}\}$ is the linear system on $X$ corresponding to the morphism $\varphi$. The support of the divisor $\varphi^{\ast}L$ is denoted by ${\rm supp}(\varphi^{\ast}L)$. For a point $P\in X$, the order of $\varphi^{\ast}L$ at $P$ is denoted by ${\rm ord}_{P}(\varphi^{\ast}L)$. We put $\alpha_{P}={\rm min}\{ {\rm ord}_{P}(\varphi^{\ast}L)\mid \varphi^{\ast}L\in \Lambda , P\in {\rm supp}(\varphi^{\ast}L)\}$. Then there exists a unique line $\widetilde{L}$ such that $\beta_{P}={\rm ord}_{P}(\varphi^{\ast}(\widetilde{L}))>\alpha_{P}$. We call the line $\widetilde{L}$ the osculating line at $P$, and we call the sequence $(0, \alpha_{P}, \beta_{P})$ the $(\Lambda, P)$-order sequence at $P$. A line $\widetilde{L}$ passing through $\varphi (P)$ is called a tangent line at $\varphi (P)$ if $\widetilde{L}$ is the osculating line at a point contained in $\varphi^{-1}(\varphi (P))$. Note that a line $\widetilde{L}$ is a tangent line at $\varphi (P)$ if and only if $m_{\varphi (P)}<I_{\varphi (P)}(\varphi (X), \widetilde{L})$, where $I_{\varphi (P)}(\varphi (X), \widetilde{L})$ is the intersection multiplicity of $\varphi (X)$ and $\widetilde{L}$ at $\varphi (P)$. 

Next, we consider the projection $\pi_{\varphi (P)}$, and we put ${\hat{\pi}}_{\varphi (P)}=\pi_{\varphi (P)}\circ \varphi :X\to \mathbb{P}^{1}$. We recall some properties of a ramification index of ${\hat{\pi}}_{\varphi (P)}$. We put $\varphi^{-1}(\varphi (P))=\{ P_{1},\dots , P_{n}\}$. Let $(0, \alpha_{P_{i}}, \beta_{P_{i}})$ be the $(\Lambda , P_{i})$-order sequence for $i=1,\dots , n$. The ramification index of ${\hat{\pi}}_{\varphi (P)}$ at a point $Q\in X$ is denoted by $e_{Q}({\hat{\pi}}_{\varphi (P)})$. Then the following fact is well-known.
\begin{fact}\label{proj}
Let $Q\in X\setminus \{ P_{1},\dots , P_{n}\}$.
\begin{itemize}
\item[(1)] The equality $e_{Q}({\hat{\pi}}_{\varphi (P)})={\rm ord}_{Q}(\varphi^{\ast}\overline{\varphi (P)\varphi (Q)})$ holds.
\item[(2)] The equality $e_{P_{i}}({\hat{\pi}}_{\varphi (P)})=\beta_{P_{i}}-\alpha_{P_{i}}$ holds for $i=1,\dots , n$.
\end{itemize}
\end{fact}

Finally, we recall some properties of a Galois covering for the proof of our main theorem (see \cite[{\rm I\hspace{-.01em}I\hspace{-.01em}I}. 7.1, 7.2 and 8.2]{Stichtenoth}).
\begin{fact}\label{Galoiscov}
Let $\theta :X\to Y$ be a surjective morphism of smooth projective curves, and let the field extension $k(X)/\theta^{\ast}k(Y)$ be a Galois extension with the Galois group $G$. Then the following hold.
\begin{itemize}
\item[(1)] If $P, Q\in X, \theta (P)=\theta (Q)$, then there exists an element $\sigma \in G$ such that $\sigma (P)=Q$.
\item[(2)] If $P, Q\in X, \theta (P)=\theta (Q)$, then $e_{P}(\theta)=e_{Q}(\theta)$.
\item[(3)] For each point $P\in X$, the order $|G(P)|$ is equal to $e_{P}(\theta)$.
\end{itemize}
\end{fact}

\section{Proof of Theorem \ref{theorem}}
The same notations are used as in the previous section. The following lemma shows that Theorem \ref{theorem} describes all cases with two inner Galois points.
\begin{lemma}\label{lemma}
Let $P_{1}, P_{2}\in X$, and let $\varphi (P_{1}), \varphi (P_{2})$ be different inner Galois points. Then $m_{\varphi (P_{1})}=I_{\varphi (P_{1})}(\varphi (X), L)$ or $m_{\varphi (P_{2})}=I_{\varphi (P_{2})}(\varphi (X), L)$ holds.
\end{lemma}

%$L=\overline{\varphi (P_{1})\varphi (P_{2})}$ is not a tangent line at $\varphi (P_{1})$ or $\varphi (P_{2})$. In other words, 

\begin{proof}
We put $\varphi^{-1}(\varphi (P_{1}))=\{ P_{11}=P_{1}, P_{12},\dots , P_{1n_{1}}\}$ and $\varphi^{-1}(\varphi (P_{2}))=\{ P_{21}=P_{2}, P_{22},\dots , P_{2n_{2}}\}$. Let $(0, \alpha_{P_{ij}}, \beta_{P_{ij}})$ be the $(\Lambda, P_{ij})$-order sequence for $i, j$. Assume by contradiction that $m_{\varphi (P_{1})}<I_{\varphi (P_{1})}(\varphi (X), L)$ and $m_{\varphi (P_{2})}<I_{\varphi (P_{2})}(\varphi (X), L)$ hold. By Fact \ref{Galoiscov}, the ramification index of ${\hat{\pi}}_{\varphi (P_{1})}$ (resp. ${\hat{\pi}}_{\varphi (P_{2})}$) at each point contained in $\varphi^{-1} (\varphi (P_{2}))$ (resp. $\varphi^{-1} (\varphi (P_{1}))$) coincides with $|G_{\varphi (P_{1})}(P_{2})|$ (resp. $|G_{\varphi (P_{2})}(P_{1})|$). By Fact \ref{proj} (1) and Fact \ref{Galoiscov},  $|G_{\varphi (P_{1})}(P_{2})|$ (resp. $|G_{\varphi (P_{2})}(P_{1})|$) coincides with ${\rm ord}_{P_{2j}}(\varphi^{\ast}L)$ for each $j$ (resp. ${\rm ord}_{P_{1i}}(\varphi^{\ast}L)$ for each $i$). Since $L$ is a tangent line at $\varphi (P_{1})$ (resp. $\varphi (P_{2})$), there exists $i_{0}$ (resp. $j_{0}$) such that $\beta_{P_{1i_{0}}}={\rm ord}_{P_{1i_{0}}}(\varphi^{\ast}L)$ (resp. $\beta_{P_{2j_{0}}}={\rm ord}_{P_{2j_{0}}}(\varphi^{\ast}L)$).  By Fact \ref{proj} (2) and Fact \ref{Galoiscov}, $|G_{\varphi (P_{1})}(P_{2})|=\beta_{P_{1i_{0}}}-\alpha_{P_{1i_{0}}}$ (resp. $|G_{\varphi (P_{2})}(P_{1})|=\beta_{P_{2j_{0}}}-\alpha_{P_{2j_{0}}}$) holds. Therefore, we have a contradiction as follows:
\begin{eqnarray*}
|G_{\varphi (P_{2})}(P_{1})|&<&|G_{\varphi (P_{2})}(P_{1})|+\alpha_{P_{2j_{0}}}=\beta_{P_{2j_{0}}}={\rm ord}_{P_{2j_{0}}}(\varphi^{\ast}L)=|G_{\varphi (P_{1})}(P_{2})|\\
&<&|G_{\varphi (P_{1})}(P_{2})|+\alpha_{P_{1i_{0}}}=\beta_{P_{1i_{0}}}={\rm ord}_{P_{1i_{0}}}(\varphi^{\ast}L)=|G_{\varphi (P_{2})}(P_{1})|.
\end{eqnarray*}
\end{proof}

\begin{proof}[{\bf Proof of Theorem \ref{theorem}}]
We consider the only-if part. Assume that conditions (a), (b), (c), and (d) of Theorem \ref{theorem} are satisfied. Let $D$ be the divisor given below:
\[ D={\rm Bs}_{P_{1}}+\sum_{\sigma \in G_{1}}\sigma (P_{2}).\]
By condition (d), the following equality of divisors holds: 
\[ D={\rm Bs}_{P_{2}}+\sum_{\tau \in G_{2}}\tau (P_{1}).\]
Let $f, g\in k(X)$ be the generators of $k(X)^{G_{1}}, k(X)^{G_{2}}$ such that $(f)_{\infty}=D-{\rm Bs}_{P_{1}}, (g)_{\infty}=D-{\rm Bs}_{P_{2}}$, by condition (a), where $(f)_{\infty}$ (resp. $(g)_{\infty}$) is the pole divisor of $f$ (resp. $g$). Then $f, g\in \mathcal{L}(D)$. We consider the morphism $\varphi=(f:g:1):X\to \mathbb{P}^{2}$. First, we show that the equality $\varphi (P_{1})=(0:1:0)$ holds. We put $n_{g}={\rm ord}_{P_{1}}((g)_{\infty})$. Note that $n_{g}$ is equal to $|G_{2}(P_{1})|$. Let $t_{P_{1}}$ be a uniformizer at $P_{1}$. If $P_{1}\not\in {\rm supp}((f)_{\infty})$, then 
\[ {\rm ord}_{P_{1}}(t_{P_{1}}^{n_{g}}f)=n_{g}+{\rm ord}_{P_{1}}(f)\geq n_{g}>0\]
holds. If $P_{1}\in {\rm supp}((f)_{\infty})$, then $P_{1}\in (G_{1}\cdot P_{2})\cap (G_{2}\cdot P_{1})$. Since the inequality ${\rm Bs}_{P_{1}}\geq P_{1}$ of divisors holds by condition (c), we have 
\[ {\rm ord}_{P_{1}}(t_{P_{1}}^{n_{g}}f)=|G_{2}(P_{1})|-|G_{1}(P_{2})|>0.\]
Therefore, the equality $\varphi (P_{1})=(0:1:0)$ holds. Since the inequality ${\rm Bs}_{P_{2}}\geq P_{2}$ of divisors holds by condition (c), we have $P_{2}\not\in G_{2}\cdot P_{1}={\rm supp}((g)_{\infty})$. Therefore, the equality $\varphi (P_{2})=(1:0:0)$ holds. Similar to the proof of \cite[Proposition 1]{FukasawaAlgebra2018}, by condition (b), we can show that the morphism $\varphi$ is birational onto its image. Since ${\rm supp}((f)+D)\cap {\rm supp}((g)+D)\cap {\rm supp}(D)=\emptyset$, the sublinear system $\Lambda$ of the complete linear system $|D|$ corresponding to $\langle f, g, 1\rangle$ is base-point-free. Therefore, the equalities ${\rm deg}(\varphi (X))={\rm deg}(D)=|G_{1}|+{\rm deg}({\rm Bs}_{P_{1}})=|G_{2}|+{\rm deg}({\rm Bs}_{P_{2}})$ hold. The morphism $(f:1)$ (resp. $(g:1)$) coincides with the projection from the point $\varphi (P_{1})=(0:1:0)$ (resp. $\varphi (P_{2})=(1:0:0)$). Finally, we show that the line $L=\overline{\varphi (P_{1})\varphi (P_{2})}$ is not a tangent line at $\varphi (P_{2})$. Assume by contradiction that $L$ is a tangent line at $\varphi (P_{2})$. Then there exists a point $R\in \varphi^{-1}(\varphi (P_{2}))$ such that $R\in G_{2}\cdot P_{1}$. Let (0, $\alpha$, $\beta$) be the ($\Lambda$, $R$)-order sequence. Note that the equality $\beta={\rm ord}_{R}(\varphi^{\ast}L)$ holds. Since the inequality ${\rm Bs}_{P_{1}}\geq P_{1}$ of divisors holds by condition (c), we have the inequality $|G_{2}(P_{1})|-|G_{1}(P_{2})|\geq 0$. On the other hand, by Fact \ref{proj} (1) and Fact \ref{Galoiscov}, the equality $|G_{1}(P_{2})|={\rm ord}_{R}(\varphi^{\ast}L)$ holds. Therefore, by Fact \ref{proj} (2) and Fact \ref{Galoiscov}, we have $|G_{1}(P_{2})|=\beta>\beta -\alpha=|G_{2}(P_{1})|$. This is a contradiction.

We consider the if part. Assume that $\varphi :X\to \mathbb{P}^{2}$ is a birational embedding of degree $|G_{1}|+{\rm deg}({\rm Bs}_{P_{1}})=|G_{2}|+{\rm deg}({\rm Bs}_{P_{2}})$ such that $\varphi (P_{1})$, $\varphi (P_{2})$ are different inner Galois points for $\varphi (X)$, and $G_{\varphi (P_{1})}, G_{\varphi (P_{2})}$ coincide with $G_{1}, G_{2}$ respectively, and $\overline{\varphi(P_{1})\varphi(P_{2})}$ is not a tangent line at $\varphi (P_{2})$. Let $\Lambda$ be the linear system corresponding to the morphism $\varphi$. We put $\varphi^{-1}(\varphi (P_{1}))=\{ P_{11}=P_{1}, P_{12},\dots , P_{1n_{1}}\}, \varphi^{-1}(\varphi (P_{2}))=\{ P_{21}=P_{2}, P_{22},\dots , P_{2n_{2}}\}$. Let $(0, \alpha_{P_{ij}}, \beta_{P_{ij}})$ be the $(\Lambda, P_{ij})$-order sequence for $i, j$. Since $k(X)^{G_{i}}=({\hat{\pi}}_{\varphi (P_{i})})^{\ast}(k(\mathbb{P}^{1}))\cong k(\mathbb{P}^{1})$ for $i=1, 2$, condition (a) is satisfied. Similar to the proof of \cite[Theorem 1]{FukasawaAlgebra2018}, condition (b) is satisfied. We put $L=\overline{\varphi (P_{1})\varphi (P_{2})}$, and we consider the divisor $D=\varphi^{\ast}L$. Since the linear system corresponding to ${\hat{\pi}}_{\varphi (P_{1})}$ (resp. ${\hat{\pi}}_{\varphi (P_{2})}$) is $\{ E-\sum_{i=1}^{n_{1}}\alpha_{P_{1i}}P_{1i}\mid E\in \Lambda \}$ (resp. $\{ E-\sum_{j=1}^{n_{2}}\alpha_{P_{2j}}P_{2j}\mid E\in \Lambda \}$) and ${\hat{\pi}}_{\varphi (P_{1})}$ (resp. ${\hat{\pi}}_{\varphi (P_{2})}$) is a Galois covering,
\[ ({\hat{\pi}}_{\varphi (P_{1})})^{\ast}(L)=D-\sum_{i=1}^{n_{1}}\alpha_{P_{1i}}P_{1i}=\sum_{\sigma \in G_{1}}\sigma (P_{2})\]
\[ ({\rm resp.}\ ({\hat{\pi}}_{\varphi (P_{2})})^{\ast}(L)=D-\sum_{j=1}^{n_{2}}\alpha_{P_{2j}}P_{2j}=\sum_{\tau \in G_{2}}\tau (P_{1}))\]
holds. Since $L$ is not a tangent line at the point $\varphi (P_{2})$, the intersection of two sets $G_{2}\cdot P_{1}$ and $\varphi^{-1} (\varphi (P_{2}))$ is the empty set. By Fact \ref{proj} (1) and Fact \ref{Galoiscov}, the equalities $|G_{1}(P_{2})|=e_{P_{2j}}({\hat{\pi}}_{\varphi (P_{1})})={\rm ord}_{P_{2j}}(\varphi^{\ast}(L))$ hold for all $j$. Therefore, the equality $\alpha_{P_{2j}}=|G_{1}(P_{2})|$ holds for all $j$, and we have the equalities
\[ D=\varphi^{\ast}L=\sum_{j=1}^{n_{2}}|G_{1}(P_{2})|P_{2j}+\sum_{\tau \in G_{2}}\tau (P_{1})={\rm Bs}_{P_{2}}+\sum_{\tau \in G_{2}}\tau (P_{1}).\] The equalities $|G_{2}(P_{1})|=e_{P_{1i}}({\hat{\pi}}_{\varphi (P_{2})})={\rm ord}_{P_{1i}}(\varphi^{\ast}(L))$ also hold for all $i$. We prove the equality $D={\rm Bs}_{P_{1}}+\sum_{\sigma \in G_{1}}\sigma (P_{2})$ in three cases:

$\rm(\hspace{.18em}i\hspace{.18em})$ We assume that $L$ is not a tangent line at $\varphi (P_{1})$ with $(L\cap \varphi (X))\setminus \{ \varphi (P_{1}), \varphi (P_{2})\}\not= \emptyset$. We take a point $Q\in \varphi^{-1}((L\cap \varphi (X))\setminus \{ \varphi (P_{1}), \varphi (P_{2})\})$. By Fact \ref{proj} (1) and Fact \ref{Galoiscov}, the equalities $|G_{2}(P_{1})|={\rm ord}_{Q}(\varphi^{\ast}(L))=|G_{1}(P_{2})|$ hold. Therefore, we have the equality
\[ (|G_{2}(P_{1})|-|G_{1}(P_{2})|)\sum_{R\in G_{1}\cdot P_{2}\cap G_{2}\cdot P_{1}}R=0.\] 
Since $L$ is not a tangent line at $\varphi (P_{1})$, the intersection of two sets $G_{1}\cdot P_{2}$ and $\varphi^{-1}(\varphi (P_{1}))$ is the empty set. Therefore, the equality $\alpha_{P_{1i}}=|G_{2}(P_{1})|$ holds for all $i$, and we have the equalities 
\[ D=\varphi^{\ast}L=\sum_{i=1}^{n_{1}}|G_{2}(P_{1})|P_{1i}+\sum_{\sigma \in G_{1}}\sigma (P_{2})={\rm Bs}_{P_{1}}+\sum_{\sigma \in G_{1}}\sigma (P_{2}).\]

$\rm(\hspace{.08em}ii\hspace{.08em})$ We assume that $L$ is not a tangent line at $\varphi (P_{1})$ with $(L\cap \varphi (X))\setminus \{ \varphi (P_{1}), \varphi (P_{2})\}= \emptyset$. Then $G_{1}\cdot P_{2}=\varphi^{-1} (\varphi (P_{2}))$, $G_{2}\cdot P_{1}=\varphi^{-1} (\varphi (P_{1}))$, and $G_{1}\cdot P_{2}\cap G_{2}\cdot P_{1}=\emptyset$ hold. Therefore, the equality $\alpha_{P_{1i}}=|G_{2}(P_{1})|$ holds for all $i$, and we have
\[ (|G_{2}(P_{1})|-|G_{1}(P_{2})|)\sum_{R\in G_{1}\cdot P_{2}\cap G_{2}\cdot P_{1}}R=0,\]
\[ D=\varphi^{\ast}L=\sum_{i=1}^{n_{1}}|G_{2}(P_{1})|P_{1i}+\sum_{\sigma \in G_{1}}\sigma (P_{2})={\rm Bs}_{P_{1}}+\sum_{\sigma \in G_{1}}\sigma (P_{2}).\]

$\rm(i\hspace{-.08em}i\hspace{-.08em}i)$ We assume that $L$ is a tangent line at $\varphi (P_{1})$. We consider two sets
\[ W=\{ P_{1i}\in \varphi^{-1}(\varphi (P_{1}))\mid 1\leq i\leq n_{1}, {\rm ord}_{P_{1i}}\varphi^{\ast}L=\beta_{P_{1i}}\},\]
\[ W'=\{ P_{1i}\in \varphi^{-1}(\varphi (P_{1}))\mid 1\leq i\leq n_{1}, {\rm ord}_{P_{1i}}\varphi^{\ast}L=\alpha_{P_{1i}}\}.\]
Note that $\varphi^{-1}(\varphi (P_{1}))=W\cup W'$, $W\not= \emptyset$, and $G_{1}\cdot P_{2}\cap \varphi^{-1}(\varphi (P_{1}))=W$ hold. For each point $P_{1i}\in W'$, the equalities $|G_{2}(P_{1})|={\rm ord}_{P_{1i}}\varphi^{\ast}L=\alpha_{P_{1i}}$ hold. On the other hand, for each point $P_{1i}\in W$, the equalities $|G_{2}(P_{1})|={\rm ord}_{P_{1i}}\varphi^{\ast}L=\beta_{P_{1i}}$ hold. By Fact \ref{proj} (2) and Fact \ref{Galoiscov}, the equality $|G_{1}(P_{2})|=\beta_{P_{1i}}-\alpha_{P_{1i}}$ holds for each point $P_{1i}\in W$. Therefore, the equality $\alpha_{P_{1i}}=|G_{2}(P_{1})|-|G_{1}(P_{2})|$ holds for each point $P_{1i}\in W$, and we have the equality
\[ \sum_{i=1}^{n_{1}}\alpha_{P_{1i}}P_{1i}=\sum_{P_{1i}\in W'}|G_{2}(P_{1})|P_{1i}+\sum_{P_{1i}\in W}(|G_{2}(P_{1})|-|G_{1}(P_{2})|)P_{1i}.\]
We show that $\mathbb{O}\setminus G_{1}\cdot P_{2}=W'$ and $G_{1}\cdot P_{2}\cap G_{2}\cdot P_{1}=W$ hold. If $(L\cap \varphi (X))\setminus \{ \varphi (P_{1}), \varphi (P_{2})\}\not= \emptyset$, then the equality $|G_{2}(P_{1})|=|G_{1}(P_{2})|$ holds. Since $|G_{2}(P_{1})|-|G_{1}(P_{2})|>0$, this is a contradiction. Therefore, $L\cap \varphi (X)=\{ \varphi (P_{1}), \varphi (P_{2})\}$. It is not difficult to check that $G_{1}\cdot P_{2}=(\varphi^{-1}(\varphi (P_{2})))\cup W$ and $G_{2}\cdot P_{1}=\varphi^{-1}(\varphi (P_{1}))$ hold, and we have the equalities
\begin{eqnarray*}
D=\varphi^{\ast}L&=&\sum_{Q\in \mathbb{O}\setminus G_{1}\cdot P_{2}}|G_{2}(P_{1})|Q+\sum_{R\in G_{1}\cdot P_{2}\cap G_{2}\cdot P_{1}}(|G_{2}(P_{1})|-|G_{1}(P_{2})|)R+\sum_{\sigma \in G_{1}}\sigma (P_{2})\\
&=&{\rm Bs}_{P_{1}}+\sum_{\sigma \in G_{1}}\sigma (P_{2}).
\end{eqnarray*}
\end{proof}

Let $\varphi$ be as in Theorem \ref{theorem}, and $\Lambda$ be the linear system on $X$ corresponding to the morphism $\varphi$. In general, for a plane curve $C$ of degree $d$ and a point $P$ on $C$, the degree of the projection with the center $P$ is $d-m_{P}$. Therefore, Proposition \ref{caution2} (1) and (2) follow directly from Theorem \ref{theorem}. By the calculations of the divisor $D=\varphi^{\ast}L$ in the proof of Theorem \ref{theorem}, Proposition \ref{caution2} (3), (4) and (5) follow.

\begin{remark}
In \cite{FukasawaarXiv}, Fukasawa presented a criterion for the existence of a birational embedding with smooth and outer Galois points. Similar to the proof of Theorem \ref{theorem}, we can extend the criterion to non-smooth and outer Galois points. The necessary and sufficient conditions for the existence of a birational embedding with inner and outer Galois points are conditions (a), (b) in Theorem \ref{theorem}, and there exist $\eta \in G_{2}$ and $P\in X$ such that
\begin{itemize}
\item[(c)'] ${\rm Bs}_{P}=|G_{2}(P)|\sum_{Q\in (G_{2}\cdot P)-(G_{1}\cdot \eta (P))}Q+(|G_{2}(P)|-|G_{1}(\eta (P))|)\sum_{R\in G_{1}\cdot \eta (P)}R\geq P$,
\item[(d)'] ${\rm Bs}_{P}+\sum_{\sigma \in G_{1}}\sigma (\eta (P))=\sum_{\tau \in G_{2}}\tau (P)$.
\end{itemize}
\end{remark}

\section{Proof of Theorem \ref{example}}
We apply Theorem \ref{theorem} to rational curves. In this case, condition (a) in Theorem \ref{theorem} is always satisfied, by L\"{u}roth's theorem. We identify ${\rm Aut}(\mathbb{P}^{1})$ with the projective linear group ${\rm PGL}(2, k)$. We put $Q_{\infty}=(1:0), Q_{a}=(a:1)\in \mathbb{P}^{1}$ for any $a\in k$.
\begin{proof}[{\bf Proof of Theorem \ref{example}}]
Let $p\not= 2, 5$, let $i\in k$ be a root of the polynomial $T^{2}+1\in k[T]$, and let $\xi$ be a primitive fifth root of unity.

(1). Let $p=3$, and let $P_{1}=Q_{\xi}, P_{2}=Q_{0}$. We consider two sets:
\[ G_{1}=\left\langle \left[ 
\begin{array}{cc}
1 & 1 \\
0 & 1
\end{array}
\right]\right\rangle \left\langle \left[
\begin{array}{cc}
1 & 0\\
0 & -1
\end{array}
\right] \right\rangle, 
G_{2}=\left\langle \left[ 
\begin{array}{cc}
\xi & 0 \\
0 & 1
\end{array}
\right]\right\rangle \left\langle \left[
\begin{array}{cc}
0 & 1\\
1 & 0
\end{array}
\right] \right\rangle.
\]
It is known that
\[ G_{1}=\left\langle \left[ 
\begin{array}{cc}
1 & 1 \\
0 & 1
\end{array}
\right]\right\rangle \rtimes \left\langle \left[
\begin{array}{cc}
1 & 0\\
0 & -1
\end{array}
\right] \right\rangle \cong {\rm AGL}(1, \mathbb{F}_{3}),
\]
\[ 
G_{2}=\left\langle \left[
\begin{array}{cc}
\xi & 0 \\
0 & 1
\end{array}
\right]\right\rangle  \rtimes \left\langle \left[
\begin{array}{cc}
0 & 1\\
1 & 0
\end{array}
\right] \right\rangle \cong {\bf D}_{5},
\]
where {\rm AGL}(1, $\mathbb{F}_{3}$) is the general affine group of degree $1$ over $\mathbb{F}_{3}$, ${\bf D}_{5}$ is the dihedral group of degree $5$ (see \cite{FaberarXiv1112.1999}). By direct computations, we have the equalities $G_{1}\cap G_{2}=\{ 1\}$,
\[ G_{1}\cdot P_{2}=\{ Q_{-1}, Q_{0}, Q_{1}\},\]
\[ G_{2}\cdot P_{1}=\{ Q_{1}, Q_{\xi}, Q_{\xi^{2}}, Q_{\xi^{3}}, Q_{\xi^{4}}\},\]
\[ \mathbb{O}=(G_{1}\cdot P_{2})\cup (G_{2}\cdot P_{1})=\{ Q_{-1}, Q_{0}, Q_{1}, Q_{\xi}, Q_{\xi^{2}}, Q_{\xi^{3}}, Q_{\xi^{4}}\},\] 
\[ G_{1}(P_{2})=\left\{ \left[
\begin{array}{cc}
1 & 0 \\
0 & 1
\end{array}
\right], \left[
\begin{array}{cc}
-1 & 0 \\
0 & 1
\end{array}
\right]
\right\},
\]
\[ G_{2}(P_{1})=\left\{ \left[
\begin{array}{cc}
1 & 0 \\
0 & 1
\end{array}
\right], \left[
\begin{array}{cc}
0 & \xi^{2} \\
1 & 0
\end{array}
\right]
\right\},
\]
\[ {\rm Bs}_{P_{1}}+\sum_{\sigma \in G_{1}}\sigma (P_{2})=(2Q_{\xi}+2Q_{\xi^{2}}+2Q_{\xi^{3}}+2Q_{\xi^{4}})+(2Q_{-1}+2Q_{0}+2Q_{1}),\]
\[ {\rm Bs}_{P_{2}}+\sum_{\tau \in G_{2}}\tau (P_{1})=(2Q_{-1}+2Q_{0})+(2Q_{1}+2Q_{\xi}+2Q_{\xi^{2}}+2Q_{\xi^{3}}+2Q_{\xi^{4}}).\]
Therefore, conditions (b), (c) and (d) in Theorem \ref{theorem} are satisfied, and there exists a birational embedding $\varphi :\mathbb{P}^{1}\to \mathbb{P}^{2}$ of ${\rm deg}(\varphi (\mathbb{P}^{1}))=14$ such that $\varphi (P_{1})$, $\varphi (P_{2})$ are different non-smooth Galois points, $G_{\varphi (P_{1})}\cong {\rm AGL}(1, \mathbb{F}_{3})$, $G_{\varphi (P_{2})}\cong {\bf D}_{5}$, and $\overline{\varphi (P_{1})\varphi (P_{2})}$ is not a tangent line at $\varphi (P_{2})$. By Proposition \ref{caution2} (1) and (2), $m_{\varphi (P_{1})}=8$ and $m_{\varphi (P_{2})}=4$. By the shape of the divisor ${\rm Bs}_{P_{1}}$, the second order at each point $P\in \varphi^{-1}(\varphi (P_{1}))$ is equal to $|G_{2}(P_{1})|=2$. Therefore, $\overline{\varphi (P_{1})\varphi (P_{2})}$ is not the osculating line at each point $P\in \varphi^{-1}(\varphi (P_{1}))$ by Proposition \ref{caution2} (3) and (4), and $\overline{\varphi (P_{1})\varphi (P_{2})}$ is not a tangent line at $\varphi (P_{1})$. 

(2). Let $P_{1}=Q_{\infty}$, $P_{2}=Q_{\xi}$. We consider
\[ G_{1}=\left\langle \left[
\begin{array}{cc}
\xi & 0 \\
0 & 1
\end{array}
\right] \right\rangle, 
G_{2}=\left\langle \left[
\begin{array}{cc}
1 & 0 \\
0 & -1
\end{array}
\right], \left[
\begin{array}{cc}
0 & 1 \\
1 & 0
\end{array}
\right] \right\rangle \left\langle \left[
\begin{array}{cc}
1 & i \\
1 & -i
\end{array}
\right] \right\rangle.
\]
Obviously, $G_{1}\cong \mathbb{Z}/5\mathbb{Z}$, and the following equation is known.
\[
G_{2}=\left\langle \left[
\begin{array}{cc}
1 & 0 \\
0 & -1
\end{array}
\right], \left[
\begin{array}{cc}
0 & 1 \\
1 & 0
\end{array}
\right] \right\rangle \rtimes \left\langle \left[
\begin{array}{cc}
1 & i \\
1 & -i
\end{array}
\right] \right\rangle \cong {\bf A}_{4},
\]
where ${\bf A}_{4}$ is the alternative group of degree $4$ (see \cite{FaberarXiv1112.1999}). Since $5$ and $12$ are coprime, condition (b) in Theorem \ref{theorem} is satisfied. By direct computations, we have the following equalities:
\[ G_{1}\cdot P_{2}=\{ Q_{1}, Q_{\xi}, Q_{\xi^{2}}, Q_{\xi^{3}}, Q_{\xi^{4}}\},\]
\[ G_{2}\cdot P_{1}=\{ Q_{-i}, Q_{-1}, Q_{0}, Q_{1}, Q_{i}, Q_{\infty}\}, \]
\[ \mathbb{O}=(G_{1}\cdot P_{2})\cup (G_{2}\cdot P_{1})=\{ Q_{-i}, Q_{-1}, Q_{0}, Q_{1}, Q_{i}, Q_{\infty}, Q_{\xi}, Q_{\xi^{2}}, Q_{\xi^{3}}, Q_{\xi^{4}}\},\] 
\[ G_{1}(P_{2})=\left\{ \left[
\begin{array}{cc}
1 & 0 \\
0 & 1
\end{array}
\right] \right\},
\]
\[ G_{2}(P_{1})=\left\{ \left[
\begin{array}{cc}
1 & 0 \\
0 & 1
\end{array}
\right], \left[
\begin{array}{cc}
1 & 0 \\
0 & -1
\end{array}
\right]
\right\},
\]
\[ {\rm Bs}_{P_{1}}+\sum_{\sigma \in G_{1}}\sigma (P_{2})=(2Q_{-i}+2Q_{-1}+2Q_{0}+2Q_{i}+2Q_{\infty}+Q_{1})+(Q_{1}+Q_{\xi}+Q_{\xi^{2}}+Q_{\xi^{3}}+Q_{\xi^{4}}),\]
\[ {\rm Bs}_{P_{2}}+\sum_{\tau \in G_{2}}\tau (P_{1})=(Q_{\xi}+Q_{\xi^{2}}+Q_{\xi^{3}}+Q_{\xi^{4}})+(2Q_{-i}+2Q_{-1}+2Q_{0}+2Q_{1}+2Q_{i}+2Q_{\infty}).\]
Therefore, conditions (c) and (d) in Theorem \ref{theorem} are satisfied, and there exists a birational embedding $\varphi :\mathbb{P}^{1}\to \mathbb{P}^{2}$ of ${\rm deg}(\varphi (\mathbb{P}^{1}))=16$ such that $\varphi (P_{1})$, $\varphi (P_{2})$ are different non-smooth Galois points, $G_{\varphi (P_{1})}\cong \mathbb{Z}/5\mathbb{Z}$, $G_{\varphi (P_{2})}\cong {\bf A}_{4}$ and $\overline{\varphi (P_{1})\varphi (P_{2})}$ is not a tangent line at $\varphi (P_{2})$. By Proposition \ref{caution2} (1) and (2), $m_{\varphi (P_{1})}=11$ and $m_{\varphi (P_{2})}=4$. By the shape of the divisor ${\rm Bs}_{P_{1}}$, the second order at the point $Q_{1}$ is equal to $|G_{2}(P_{1})|-|G_{1}(P_{2})|=1$. Therefore, $\overline{\varphi (P_{1})\varphi (P_{2})}$ is the osculating line at $Q_{1}$ by Proposition \ref{caution2} (3) and (4), and $\overline{\varphi (P_{1})\varphi (P_{2})}$ is a tangent line at $\varphi (P_{1})$.

(3). Let $P_{1}=Q_{\infty}$, $P_{2}=Q_{\xi}$. We consider two groups:
\[ G_{1}=\left\langle \left[
\begin{array}{cc}
\xi & 0 \\
0 & 1
\end{array}
\right] \right\rangle, 
G_{2}=\left\langle \left\langle \left[
\begin{array}{cc}
1 & 0 \\
0 & -1
\end{array}
\right], \left[
\begin{array}{cc}
0 & 1 \\
1 & 0
\end{array}
\right] \right\rangle \rtimes \left\langle \left[
\begin{array}{cc}
1 & i \\
1 & -i
\end{array}
\right] \right\rangle, \left\langle \left[
\begin{array}{cc}
i & 0 \\
0 & 1
\end{array}
\right] \right\rangle \right\rangle.
\]
Obviously, $G_{1}\cong \mathbb{Z}/5\mathbb{Z}$, and the following equation is known:
\[
\left\langle \left[
\begin{array}{cc}
1 & 0 \\
0 & -1
\end{array}
\right], \left[
\begin{array}{cc}
0 & 1 \\
1 & 0
\end{array}
\right] \right\rangle \rtimes \left\langle \left[
\begin{array}{cc}
1 & i \\
1 & -i
\end{array}
\right] \right\rangle \triangleleft G_{2}\cong {\bf S}_{4},
\]
where ${\bf S}_{4}$ is the symmetric group of degree $4$ (see \cite{FaberarXiv1112.1999}). Since $5$ and $24$ are coprime, condition (b) in Theorem \ref{theorem} is satisfied. By direct computations, we have the following equalities:
\[ G_{1}\cdot P_{2}=\{ Q_{1}, Q_{\xi}, Q_{\xi^{2}}, Q_{\xi^{3}}, Q_{\xi^{4}}\},\]
\[ G_{2}\cdot P_{1}=\{ Q_{-i}, Q_{-1}, Q_{0}, Q_{1}, Q_{i}, Q_{\infty}\}, \]
\[ \mathbb{O}=(G_{1}\cdot P_{2})\cup (G_{2}\cdot P_{1})=\{ Q_{-i}, Q_{-1}, Q_{0}, Q_{1}, Q_{i}, Q_{\infty}, Q_{\xi}, Q_{\xi^{2}}, Q_{\xi^{3}}, Q_{\xi^{4}}\},\] 
\[ G_{1}(P_{2})=\left\{ \left[
\begin{array}{cc}
1 & 0 \\
0 & 1
\end{array}
\right] \right\},
\]
\[ G_{2}(P_{1})=\left\{ \left[
\begin{array}{cc}
1 & 0 \\
0 & 1
\end{array}
\right], \left[
\begin{array}{cc}
1 & 0 \\
0 & -1
\end{array}
\right], \left[
\begin{array}{cc}
i & 0 \\
0 & 1
\end{array}
\right], \left[
\begin{array}{cc}
i & 0 \\
0 & -1
\end{array}
\right]
\right\},
\]
\[ {\rm Bs}_{P_{1}}+\sum_{\sigma \in G_{1}}\sigma (P_{2})=(4Q_{-i}+4Q_{-1}+4Q_{0}+4Q_{i}+4Q_{\infty}+3Q_{1})+(Q_{1}+Q_{\xi}+Q_{\xi^{2}}+Q_{\xi^{3}}+Q_{\xi^{4}}),\]
\[ {\rm Bs}_{P_{2}}+\sum_{\tau \in G_{2}}\tau (P_{1})=(Q_{\xi}+Q_{\xi^{2}}+Q_{\xi^{3}}+Q_{\xi^{4}})+(4Q_{-i}+4Q_{-1}+4Q_{0}+4Q_{1}+4Q_{i}+4Q_{\infty}).\]
Therefore, conditions (c) and (d) in Theorem \ref{theorem} are satisfied, and there exists a birational embedding $\varphi :\mathbb{P}^{1}\to \mathbb{P}^{2}$ of ${\rm deg}(\varphi (\mathbb{P}^{1}))=28$ such that $\varphi (P_{1})$, $\varphi (P_{2})$ are different non-smooth Galois points, $G_{\varphi (P_{1})}\cong \mathbb{Z}/5\mathbb{Z}$, $G_{\varphi (P_{2})}\cong {\bf S}_{4}$, and $\overline{\varphi (P_{1})\varphi (P_{2})}$ is not a tangent line at $\varphi (P_{2})$. By Proposition \ref{caution2} (1) and (2), $m_{\varphi (P_{1})}=23$ and $m_{\varphi (P_{2})}=4$. By the shape of the divisor ${\rm Bs}_{P_{1}}$, the second order at the point $Q_{1}$ is equal to $|G_{2}(P_{1})|-|G_{1}(P_{2})|=3$. Therefore, $\overline{\varphi (P_{1})\varphi (P_{2})}$ is the osculating line at $Q_{1}$ by Proposition \ref{caution2} (3) and (4), and $\overline{\varphi (P_{1})\varphi (P_{2})}$ is a tangent line at $\varphi (P_{1})$.
\end{proof}

\begin{remark}\label{caution3}
By the calculation of the divisor in the above proof, we have the second or the third order at each point contained in ${\rm supp}(\varphi^{\ast}\overline{\varphi (P_{1})\varphi (P_{2})})$. We use the same notations in the proof of Theorem \ref{example}.
\begin{itemize}
\item[(a)] For the curve (1) in the proof of Theorem \ref{example}, the second order is equal to $2$ at each point $Q\in \mathbb{O}$.
\item[(b)] For the curve (2) in the proof of Theorem \ref{example}, the second order is equal to $2$ (resp. $1$) at each point $Q\in \mathbb{O}\setminus G_{1}\cdot P_{2}$ (resp. $Q\in G_{1}\cdot P_{2}$), and the third order is equal to $2$ at $Q_{1}$.
\item[(c)] For the curve (3) in the proof of Theorem \ref{example}, the second order is equal to $4$ (resp. 3, 1) at each point $Q\in \mathbb{O}\setminus G_{1}\cdot P_{2}$ (resp. at $Q_{1}$, at each point $Q\in (G_{1}\cdot P_{2})\setminus \{ Q_{1}\}$), and the third order is equal to $4$ at $Q_{1}$.
\end{itemize}
\end{remark}

\section*{{\bf Acknowledgments}}
The author would like to thank Professor Satoru Fukasawa for the helpful discussions and invaluable comments.

\end{document}